\documentclass[12pt]{amsart}

\usepackage{amscd,latexsym,amsthm,amsfonts,amssymb,amsmath,amsxtra,color}
\usepackage[mathscr]{eucal}
\usepackage{hyperref}
\usepackage{pdfsync}
\usepackage[backgroundcolor=green!40, linecolor=green!40]{todonotes}
\newcommand{\BC}{{\mathbb {C}}}

\newcommand{\BJ}{{\mathbf {J}}}
\newcommand{\BK}{{\mathbf {K}}}

\newcommand{\BZ}{{\mathbb {Z}}}

\makeatletter
\def\Ddots{\mathinner{\mkern1mu\raise\p@
\vbox{\kern7\p@\hbox{.}}\mkern2mu
\raise4\p@\hbox{.}\mkern2mu\raise7\p@\hbox{.}\mkern1mu}}
\makeatother

\newcommand{\CB}{{\mathcal {B}}}

\newcommand{\CF}{{\mathcal {F}}}

\newcommand{\CH}{{\mathcal {H}}}

\newcommand{\CJ}{{\mathcal {J}}}

\newcommand{\CP}{{\mathcal {P}}}

\newcommand{\FA}{{\mathfrak {A}}}
\newcommand{\FB}{{\mathfrak {B}}}

\newcommand{\FK}{{\mathfrak {K}}}
\newcommand{\FL}{{\mathfrak {L}}}

\newcommand{\FO}{{\mathfrak {O}}}
\newcommand{\FP}{{\mathfrak {P}}}

\newcommand{\Fo}{{\mathfrak {o}}}
\newcommand{\Fp}{{\mathfrak {p}}}

\newcommand{\RB}{{\mathrm {B}}}

\newcommand{\RG}{{\mathrm {G}}}

\newcommand{\RP}{{\mathrm {P}}}

\newcommand{\RU}{{\mathrm {U}}}

\newcommand{\RZ}{{\mathrm {Z}}}

\newcommand{\ScB}{{\mathscr {B}}}

\newcommand{\Aut}{{\operatorname{Aut}}}

\newcommand{\der}{{\operatorname{der}}}

\newcommand{\End}{{\operatorname{End}}}

\newcommand{\GL}{{\operatorname{GL}}}

\newcommand{\Hom}{{\operatorname{Hom}}}

\newcommand{\Ind}{{\operatorname{Ind}}}
\newcommand{\ind}{{\operatorname{ind}}}

\newcommand{\Supp}{{\operatorname{Supp}}}

\newcommand{\tr}{{\operatorname{tr}}}

\newcommand{\incl}{\hookrightarrow}

\def\bet{{\beta}}

\def\boI{{\boldsymbol 1}}

\def\eps{{\varepsilon}}

\def\gam{{\gamma}}

\def\rad{{\rm rad}}

\newtheorem{cor}[equation]{Corollary}
\newtheorem{lem}[equation]{Lemma}
\newtheorem{prop}[equation]{Proposition}
\newtheorem{conj}[equation]{Conjecture}

\newtheorem{ques/conj}[equation]{Question/Conjecture}

\newtheorem{defn}[equation]{Definition}

\newtheorem*{Jconj}{Conjecture~$\bJ(\bN,\br)$}
\newtheorem*{Jconjo}{Conjecture~$\bJ_{\boldsymbol 0}(\bN,\br)$}

\makeatletter

\newcommand{\Rmnum}[1]{\expandafter\@slowromancap\romannumeral #1@}
\makeatother

\def\ignore#1{\relax}

\def\bJ{{\boldsymbol{\CJ}}}
\def\bN{{\boldsymbol{N}}}
\def\br{{\boldsymbol{r}}}

\begin{document}
\numberwithin{equation}{section}

\title[Local Converse Problem]{On the Jacquet Conjecture on the Local Converse Problem for $p$-adic $\GL_n$}

\author{Moshe Adrian}
\address{Department of Mathematics\\
University of Toronto\\
Toronto, ON, Canada M5S 2E4}
\email{madrian@math.toronto.edu}

\author{Baiying Liu}
\address{Department of Mathematics\\
University of Utah\\
Salt Lake City, UT 84112, U.S.A.}
\email{liu@math.utah.edu}

\author{Shaun Stevens}
\address{School of Mathematics\\
University of East Anglia\\
Norwich Research Park, Norwich, NR4 7TJ, UK}
\email{Shaun.Stevens@uea.ac.uk}

\author{Peng Xu}
\address{Mathematics Institute\\
University of Warwick\\
Coventry, CV4 7AL, UK}
\email{Peng.Xu@warwick.ac.uk}

\begin{abstract}
Based on previous results of Jiang, Nien and the third author, we
prove that any two \emph{minimax} unitarizable supercuspidals
of~$\GL_N$ that have the same depth and central character admit a
\emph{special pair} of Whittaker functions.  This result gives a new
reduction towards a final proof of Jacquet's conjecture on the local
converse problem for~$\GL_N$. As a corollary of our result, we
prove Jacquet's conjecture for~$\GL_N$, when~$N$ is prime.
\end{abstract}

\date{\today}
\thanks{SS and PX were supported by the Engineering and Physical Sciences
  Research Council (grant EP/H00534X/1). BL is supported in part by a postdoc research funding from Department of Mathematics, University of Utah, and in part by Dan Ciubotaru's NSF Grants DMS-1302122}
\maketitle


\section{Introduction}\label{intro}

In the representation theory of a group~$G$, one of the basic problems
is to characterize its irreducible representations up to isomorphism.
If~$G$ is the group of points of a reductive algebraic group defined
over a non-archimedean local field~$F$, there are many invariants that
one can attach to a representation~$\pi$ of~$G$, some of which are the
central character and depth.  Capturing all of these invariants,
however, is a family of complex functions, invariants themselves,
called the \emph{local gamma factors} of~$\pi$.

Now let~$\RG_N:=\GL_N(F)$ and let~$\pi$ be an
irreducible generic representation of~$\RG_N$. The family
of local gamma factors~$\gamma(s, \pi \times \tau, \psi)$, for~$\tau$
an irreducible generic representation of~$\RG_r$, can be
defined using Rankin-Selberg convolution~\cite{JPSS83} or the
Langlands-Shahidi method~\cite{S84}. Jacquet has formulated the
following conjecture on precisely which family of local gamma factors
should uniquely determine~$\pi$.

\begin{conj}[The Jacquet Conjecture on the Local Converse Problem]\label{lcp1}
Let~$\pi_1,\pi_2$ be irreducible generic representations
of~$\RG_N$. If
\[
\gamma(s, \pi_1 \times \tau, \psi) = \gamma(s, \pi_2 \times \tau,
\psi),
\]
as functions of the complex variable~$s$, for all irreducible
generic representations~$\tau$ of~$\RG_r$ with~$r = 1, \ldots,
[\frac{N}{2}]$, then~$\pi_1 \cong \pi_2$.
\end{conj}

We refer to the introductions of~\cite{Ch06} and~\cite{JNS13} for more
related discussions on the previous known results on this
conjecture. Moreover, in general, from the discussion
in~\cite[Sections~5.2 and~5.3]{Ch96}, it is expected that the upper
bound~$[\frac{N}{2}]$ is sharp.

In~\cite[Section~2.4]{JNS13}, Conjecture~\ref{lcp1} is shown to be
equivalent to the same conjecture with the adjective ``generic''
replaced by ``unitarizable supercuspidal'' (recall that an irreducible
representation is \emph{supercuspidal} if it is not a subquotient of a
properly parabolically induced representation, while all supercuspidal
representations are generic). However, in the situation
that~$\pi_1,\pi_2$ are both supercuspidal, it may be that the upper
bound~$[\frac{N}{2}]$ is no longer sharp, at least within certain
families of supercuspidals: for example, for \emph{simple}
supercuspidals (of depth~$\frac 1N$), the upper bound may be lowered
to~$1$ (see~\cite[Proposition 2.2]{BH14} and~\cite[Remark
3.18]{AL14} in general, and~\cite{X13} in the tame case).

Thus, for~$m\ge 1$ an integer, it is convenient for us to say that
irreducible supercuspidal representations~$\pi_1,\pi_2$ of~$\RG_N$
\emph{satisfy hypothesis~$\CH_m$} if
\begin{itemize}
\item[{($\CH_m$)\ \ }]
$\gam(s,\pi_1\times\tau,\psi)=\gam(s,\pi_2\times\tau,\psi)$ as functions
of the complex variable~$s$, for all irreducible supercuspidal
representations~$\tau$ of~$\RG_m$.
\end{itemize}
For~$r\ge 1$, we say that~$\pi_1,\pi_2$ \emph{satisfy
  hypothesis~$\CH_{\le r}$} if they satisfy hypothesis~$\CH_m$,
for~$1\le m\le r$. Then we can state a family of ``conjectures''.

\begin{Jconj}
If~$\pi_1,\pi_2$ are irreducible supercuspidal representations
of~$\RG_N$ which satisfy hypothesis~$\CH_{\le r}$, then~$\pi_1\simeq\pi_2$.
\end{Jconj}

Thus Jacquet's conjecture is (equivalent to) Conjecture~$\CJ(N,[\frac N2])$, while
Conjecture~$\CJ(N,N-2)$ is a Theorem due to Chen~\cite{Ch96,Ch06} and
to Cogdell and Piatetskii-Shapiro~\cite{CPS99}.
On the other hand, examples in~\cite{Ch96}
also show that Conjecture~$\CJ(4,1)$ is false, and similar examples
show that~$\CJ(N,1)$ is false when~$N$ is composite.

Here we prove Conjecture~$\CJ(N,[\frac N2])$ when~$N$ is prime (see
Corollary~\ref{cor:Jprime}), and make some further elementary
reductions in general. This builds on the work in~\cite{JNS13}, where
many cases are proved. Indeed, our work here is to tackle the
``simplest'' case left out in~\cite{JNS13}, which is sufficient
when~$N$ is prime. We hope that this will be the first step in an
inductive proof allowing all~$N$ to be treated.

On the other hand, it is not clear to us whether
conjecture~$\CJ(N,[\frac N2])$ is optimal, when~$N$ is prime. For
example, we do not know whether conjecture~$\CJ(5,1)$ is true: is it
possible to distinguish supercuspidal representations of~$\GL_5(F)$
only by the local gamma factors of their twists by characters? If so,
it would be tempting to suggest, more generally, that two irreducible
supercuspidal representations of~$\RG_N$ which satisfy
hypothesis~$\CH_m$, for all~$m<N$ \emph{dividing}~$N$, should be
equivalent. Our methods here do not shed light on this question.

Finally we describe the contents of the paper and the scheme of the
proof. In~\cite{JNS13}, Jiang, Nien and the third author introduced
the notion of a \emph{special pair of Whittaker functions} for a pair
of irreducible unitarizable supercuspidal representations~$\pi_1,\pi_2$
of~$\RG_N$ (see Section~\ref{S:special}). They also proved that, if there
is such a pair and~$\pi_1,\pi_2$ satisfy hypothesis$~\CH_{\le [N/2]}$,
then~$\pi_1,\pi_2$ are equivalent, as well as finding special pairs of
Whittaker functions in many cases, in particular the case of depth
zero representations.

The simplest case omitted in~\cite{JNS13} is a case we
call~\emph{minimax}, described as follows. Each supercuspidal
representation~$\pi_i$ is irreducibly induced from a representation of
a compact-mod-centre subgroup, called an \emph{extended maximal simple
  type}~\cite{BK93}; amongst the data from which this is built, is a
\emph{simple stratum}~$[\FA,n,0,\beta]$ and we say the representation is minimax
if the field extension~$F[\beta]/F$ has degree~$N$ and the
element~$\beta$ is \emph{minimal} in the sense
of~\cite[(1.4.14)]{BK93} (see Section~\ref{S:strata} for
recollections).

After preparing the ground in Sections~\ref{S:unip}, \ref{S:minimax}, we
prove that any pair of minimax unitarizable supercuspidal
representations of~$\RG_N$ with the same (positive) depth and central
character possesses a special pair of Whittaker functions (see
Proposition~\ref{prop:minimax}).

Finally, when~$N$ is prime, any irreducible supercuspidal
representation is a twist by some character of either a depth zero
representation or of a minimax supercuspidal representation. Then
Jacquet's conjecture follows from the results in~\cite{JNS13} and a
reduction to representations which are of minimal depth among their
twists (see Section~\ref{S:twist}).
%

\subsection{Notation}

Throughout,~$F$ is a locally compact nonarchimedean local field, with
ring of integers~$\Fo_F$, maximal ideal~$\Fp_F$, and residue field~$k_F$
of cardinality~$q_F$ and characteristic~$p$; we also write~$\nu_F$ for the normalized valuation on~$F$, with image~$\BZ$, and~$|\cdot|$ for
the normalized absolute value on~$F$, with image~$q_F^{\BZ}$. We use
similar notation for finite extensions of~$F$.
We fix once and for all an additive
character~$\psi_F$ of~$F$ which is trivial on~$\Fp_F$ and nontrivial
on~$\Fo_F$.

For~$r\ge 1$, we set~$\RG_r=\GL_r(F)$, and denote by~$\RU_r$ the unipotent radical of the standard Borel
subgroup~$\RB_r$ of~$\RG_r$, consisting of upper-triangular
matrices. We denote by~$\psi_r$ the standard nondegenerate character of~$\RU_r$, given by
\[
\psi_r(u)=\psi_F\left(\sum_{i=1}^{r-1} u_{i,i+1}\right),
\]
where~$(u_{ij})$ is the matrix of~$u\in\RU_r$. We also denote by~$\RZ_r$ the centre of~$\RG_r$, and by~$\RP_r$ the standard mirabolic subgroup consisting of matrices with last row equal
to~$(0,\ldots,0,1)$.

We fix an integer~$N\ge 2$ and abbreviate~$\RG=\RG_N$. We also put~$V=F^N$ and~$A=\End_F(V)$, and identify~$\RG$ with~$\Aut_F(V)$ via the standard basis of~$F^N$.

All representations considered are smooth representations with complex coefficients.

\section{Special pairs of Whittaker functions}\label{S:special}

In this section, we recall the main results on special pairs of Whittaker functions. For further background, we refer to~\cite{JNS13} and the references therein.

Let~$\pi$ be an irreducible supercuspidal representation of~$\RG$. By the existence and uniqueness of local Whittaker models,~$\Hom_{\RG}(\pi,\Ind_{\RU_N}^{\RG}\psi_N)$ is a one-dimensional space. A \emph{Whittaker function} for~$\pi$ is any function~$W_\pi$ in~$\Ind_{\RU_N}^{\RG}\psi_N$ which is in the image of~$\pi$ under a non-zero homomorphism in this~$\Hom$-space. In~\cite{JNS13}, the following definitions were introduced.

\begin{defn} Let~$\pi$ be an irreducible unitarizable supercuspidal representation of~$\RG$ and let~$\BK$ be a compact-mod-centre open subgroup of~$\RG$. A nonzero Whittaker function~$W_{\pi}$ for~$\pi$ is called~\emph{$\BK$-special} if the support of~$W_{\pi}$ satisfies~$\Supp (W_{\pi}) \subset \RU_N \BK$, and if
\[
W_{\pi}(k^{-1})=\overline{W_{\pi}(k)} \text{ for all } k \in \BK,
\]
where~$\overline{z}$ denotes the complex conjugate of~$z \in \BC$.
\end{defn}

\begin{defn}
For~$i=1,2$, let~$\pi_i$ be an irreducible unitarizable supercuspidal representation of~$\RG$ and let~$W_{\pi_i}$ be a nonzero Whittaker function for $\pi_i$. Suppose moreover that~$\pi_1,\pi_2$ have the  same central character. Then~$(W_{\pi_1}, W_{\pi_2})$ is called a \emph{special pair} of Whittaker functions for the pair~$(\pi_1, \pi_2)$ if there exists a compact-mod-centre open subgroup~$\BK$ of~$\RG$ such that~$W_{\pi_1}$ and~$W_{\pi_2}$ are both~$\BK$-special and
\[
W_{\pi_1}(p)=W_{\pi_2}(p), \text{ for all } p \in P_N.
\]
\end{defn}

The condition in this definition that the representations have the
same central character is rather mild in our situation since,
by~\cite[Corollary~2.7]{JNS13}, if~$\pi_1,\pi_2$ are irreducible
supercuspidal representations of~$\RG$ which satisfy
hypothesis~$\CH_1$, then they have the same central character.

The following is one of the main results in~\cite{JNS13}, which provides a general approach to proving Conjecture~$\CJ(N,[\frac N2])$.

\begin{prop}[{\cite[Theorem~1.5]{JNS13}}]\label{prop:SPimpliesJ}
Let~$\pi_1,\pi_2$ be irreducible unitarizable supercuspidal representations
of~$\RG$ which have a special pair of Whittaker functions and
satisfy hypothesis~$\CH_{\le [N/2]}$. Then~$\pi_1\simeq\pi_2$.
\end{prop}

In~\cite{JNS13}, in several cases it is proved that a pair of supercuspidal
representations of~$\RG$ have a special pair of Whittaker
functions. Here we prove another case, the simplest case left open in~\cite{JNS13}; as we will see, this is sufficient to prove Conjecture~$\CJ(N,[\frac N2])$ in the case that~$N$ is prime.

\section{Strata}\label{S:strata}

In order to use Proposition~\ref{prop:SPimpliesJ}, we need to recall some parts of the construction of supercuspidal representations in~\cite{BK93}, in particular the notion of a stratum.

We begin with a hereditary~$\Fo_F$-order~$\FA$ in~$A=\End_F(V)$, with Jacobson radical~$\FP$, and we denote by~$e=e(\FA|\Fo_F)$ the~$\Fo_F$-period of~$\FA$, that is, the integer such that~$\Fp_F\FA=\FP^e$. For any such hereditary order~$\FA$, there is an ordered basis with respect to which~$\FA$ is in \emph{standard form}, that is, it consists of matrices with coefficients in~$\Fo_F$ which are block upper-triangular modulo~$\Fp_F$. Such a basis can be found as follows.

Recall that an~$\Fo_F$-lattice chain in~$V$ is a set of~$\Fo_F$-lattices which is linearly ordered by inclusion and invariant under multiplication by scalars in~$F^\times$. Then there is a unique~$\Fo_F$-lattice chain~$\FL=\{L_i\mid i\in\BZ\}$ in~$V$ such that~$\FA=\{x\in A\mid x L_i\subseteq L_i\text{ for all }i\in\BZ\}$. (The set~$\FL$ is uniquely determined by~$\FA$, though the base point~$L_0$ for the indexing is arbitrary.) For~$i=0,\ldots,e-1$, we choose an ordered set~$\CB_i$ of vectors in~$L_i$ whose image in~$L_i/L_{i+1}$ is a basis and then the ordered basis obtained by concatenating~$\CB_0,\ldots,\CB_{e-1}$ is as required.

A hereditary order~$\FA$ gives rise to a parahoric subgroup~$U(\FA)=U^0(\FA)=\FA^\times$ of~$\RG$, together with a filtration by normal open subgroups~$U^n(\FA)=1+\FP^n$, for~$n\ge 1$, as well as a compact-mod-centre subgroup~$\FK(\FA)=\{g\in\RG\mid g\FA g^{-1}=\FA\}$, the normalizer of~$\FA$ in~$\RG$. We also get a ``valuation''~$\nu_\FA$ on~$A$ by~$\nu_\FA(x)=\sup\{n\in\BZ\mid x\in\FP^n\}$, and, for~$x\in F$, we have~$\nu_\FA(x)=e(\FA|\Fo_F)\nu_F(x)$.

A \emph{stratum} in~$A$ is a quadruple~$[\FA,n,r,\bet]$, where~$\FA$ is a hereditary~$\Fo_F$-order,~$n\ge r\ge 0$ are integers, and~$\bet\in\FP^{-n}$. Strata~$[\FA,n,r,\bet_i]$, with~$i=1,2$, are called \emph{equivalent} if~$\bet_1-\bet_2\in\FP^{-r}$. Thus, when~$r\ge \left[\frac n2\right]$, the equivalence class of a stratum~$[\FA,n,r,\bet]$ corresponds to a character~$\psi_\bet$ of~$U^{r+1}(\FA)$, trivial on~$U^{n+1}(\FA)$, via
\[
\psi_\bet(x)=\psi_F\circ\tr_{A/F}(\bet(x-1)), \quad\text{for~$x\in U^{r+1}(\FA)$}.
\]

The stratum~$[\FA,n,r,\bet]$ is called~\emph{pure} if~$E=F[\bet]$ is a field with~$E^\times\subseteq\FK(\FA)$, and~$\nu_\FA(\bet)=-n$. A pure stratum~$[\FA,n,r,\bet]$ is called~\emph{simple} if~$r<-k_0(\bet,\FA)$, where~$k_0(\bet,\FA)$ is an invariant whose definition we do not recall here (see~\cite[Definition~1.4.5]{BK93}).

Of particular importance will be simple strata of the form~$[\FA,n,n-1,\bet]$. A pure such stratum is simple if and only if the element~$\bet$ is \emph{minimal} in the sense of~\cite[(1.4.14)]{BK93},
that is:
\begin{enumerate}
\item $\nu_E(\bet)$ is coprime to the ramification index~$e=e(E/F)$; and
\item if $\varpi_F$ is any uniformizer of~$F$, then the image of~$\varpi_F^{-\nu_E(\bet)}\bet^e + \mathfrak{p}_E$ in~$k_E$ generates the residue class extension~$k_E/k_F$.
\end{enumerate}
We call a simple stratum of the form~$[\FA,n,n-1,\bet]$ a
\emph{minimal} stratum. A particular case occurs when in fact~$\bet\in
F$; in this case we call the stratum~$[\FA,n,n-1,\bet]$ a
\emph{scalar} stratum

Finally, we call a pure stratum~$[\FA,n,r,\bet]$ a \emph{max} stratum
if the extension~$E=F[\bet]/F$ is maximal in~$A$ (that is, of
degree~$N$), in which case~$\nu_\FA$ coincides with~$\nu_E$ on~$E$
and~$e(\FA|\Fo_F)=e(E/F)$. We call a max simple
stratum~$[\FA,n,n-1,\bet]$ a \emph{minimax} stratum, in which
case~$n,e(\FA|\Fo_F)$ are coprime.

\medskip

The first step in the construction and classification of the positive depth supercuspidal representations of~$\RG$ in~\cite{BK93} is to prove that any such representation~$\pi$ \emph{contains} a simple stratum~$[\FA,n,n-1,\bet]$, in the sense that~$\Hom_{U^{n}(\FA)}(\psi_\bet,\pi)\ne 0$. The \emph{depth}~$\ell(\pi)$ of~$\pi$ is then the depth~$n/e(\FA|\Fo_F)$ of any such simple stratum; this is independent of any choices, as is the degree of the extension~$E=F[\bet]/F$. In particular, if~$\pi$ contains a minimax stratum~$[\FA,n,n-1,\bet]$, then~$n$ and~$e=e(E/F)$ are determined by the depth~$\ell(\pi)=n/e$, since they are coprime, as is the residue class degree~$f=f(E/F)=N/e$.  For more details, see for example \cite[Proposition 1.14]{KM90}.

\section{Twisting by characters}\label{S:twist}

In this section, we give a modest reduction of Conjecture~$\CJ(N,r)$
to supercuspidal representations which are of minimal depth amongst
all representations obtained from them by twisting by a character.

For~$\pi$ an irreducible representation of~$\RG$ and~$\chi$ a character of~$F^\times$,
we write~$\pi\chi$ for the representation~$\pi\otimes\chi\circ\det$ of~$\RG$.
We say that an irreducible supercuspidal representation~$\pi$ of~$\RG$ is \emph{of minimal depth in its twist class} if
\[
\ell(\pi\chi)\ge \ell(\pi), \text{ for all characters~$\chi$ of~$F^\times$}.
\]
The representation~$\pi$ is of minimal depth in its twist class if and only if it does \emph{not} contain a scalar minimal stratum;
that is, if~$[\FA,n,n-1,\bet]$ is a minimal stratum contained in~$\pi$ then~$F[\bet]$ is a proper extension of~$F$ (see \cite[Remark 1.15]{KM90} and \cite[Lemma 2.4.11]{BK93}).
Note that, in the case that~$N$ is prime, this implies that any
minimal stratum contained in~$\pi$ is a minimax stratum. (One can also
see this directly from the classification of Carayol~\cite{C84} -- see
also~\cite[p209]{B87}.)

We will not recall here the definitions of local factors of pairs of supercuspidal representations from~\cite{JPSS83}. However, from the definitions (see also~\cite[Theorem~2.7]{JPSS83}), a straightforward check shows the following, which is surely well-known.

\begin{lem}\label{lem:transfertwist}
Let~$r$ be a natural number with~$r<N$, let~$\pi,\tau$ be generic irreducible representations of~$\RG,\RG_r$ respectively, let~$\chi$ be a character of~$F^\times$, and let~$s\in\BC$. Then
\begin{eqnarray*}
L(s,\pi\chi\times\tau)&=&L(s,\pi\times\tau\chi), \\
\eps(s,\pi\chi\times\tau,\psi)&=&\eps(s,\pi\times\tau\chi,\psi), \\
\gam(s,\pi\chi\times\tau,\psi)&=&\gam(s,\pi\times\tau\chi,\psi).
\end{eqnarray*}
\end{lem}

We also recall that the depth~$\ell(\pi)$ of an irreducible
supercuspidal representation can be determined from the conductor
of the standard epsilon
factor~$\eps(s,\pi,\psi)=\eps(s,\pi\times\boI,\psi)$, where~$\boI$ is
the trivial representation of~$\RG_1$ (see~\cite{B87}); indeed the same
is true for an arbitrary discrete series representation~$\pi$,
by~\cite[Theorem 3.1]{LR03}.

Now we can reduce Conjecture~$\CJ(N,r)$ to the following special case:

\begin{Jconjo}
If~$\pi_1,\pi_2$ are irreducible supercuspidal representations
of~$\RG$ of minimal depth in their twist class which
satisfy hypothesis~$\CH_{\le r}$, then~$\pi_1\simeq\pi_2$.
\end{Jconjo}

\begin{prop}\label{prop:redmin}
For~$1\le r<N$, Conjecture~$\CJ_0(N,r)$ is equivalent to Conjecture~$\CJ(N,r)$.
\end{prop}

\begin{proof} It is clear that Conjecture~$\CJ(N,r)$ implies
Conjecture~$\CJ_0(N,r)$. For the converse, we assume that
Conjecture~$\CJ_0(N,r)$ is true, and let~$\pi_1,\pi_2$ be irreducible
supercuspidal representations
of~$\RG$ which satisfy hypothesis~$\CH_{\le r}$. Then, for~$\chi$ any character
of~$F^\times$ and~$\tau$ any supercuspidal representation of~$\RG_m$
with~$1\le m\le r$, Lemma~\ref{lem:transfertwist} and property~$\CH_{\le r}$
imply that
\begin{equation}\label{eqn:gamtwist}
\left.\begin{array}{rclll}
\gam(s,\pi_1\chi\times\tau,\psi)& =& \gam(s,\pi_1\times\tau\chi,\psi)&& \\
&=&\gam(s,\pi_2\times\tau\chi,\psi)
& =& \gam(s,\pi_2\chi\times\tau,\psi).
\end{array}\right.
\end{equation}
In particular, using the case~$m=1$ with~$\tau=\boI$ the trivial
representation, this implies that~$\ell(\pi_1\chi)=\ell(\pi_2\chi)$.

Now we pick a character~$\chi$ of~$F^\times$ such that~$\pi_1\chi$ is of
minimal depth in its twist class, that is, such
that~$\ell(\pi_1\chi)$ is minimal in~$\{\ell(\pi_1\chi)\mid\chi\text{
  a character of }F^\times\}$. Then~$\pi_2\chi$ is also of minimal
depth in its twist class
and~\eqref{eqn:gamtwist} now implies that the
representations~$\pi_1\chi,\pi_2\chi$ satisfy hypothesis~$\CH_{\le r}$. Thus, by the
assumption that Conjecture~$\CJ_0(N,r)$ is true, we deduce
that~$\pi_1\chi\simeq\pi_2\chi$, whence~$\pi_1\simeq\pi_2$ as required.
\end{proof}

\section{Unipotent and mirabolic subgroups}\label{S:unip}

Although we have fixed standard mirabolic and maximal unipotent subgroups, it will be convenient in the sequel to allow these to vary, working in the basis-free setting of Section~\ref{S:strata}. Thus, in this section, we gather some notation for arbitrary mirabolic and maximal unipotent subgroups.

A maximal flag in~$V$
\[
\CF = \{ 0=V_0\subset V_1\subset \cdots \subset V_{N-1}\subset V_N=V\},
\]
with~$\dim_F(V_i)=i$, determines both a maximal unipotent subgroup~$U_\CF$ and a mirabolic subgroup~$P_\CF$ by
\begin{eqnarray*}
U_\CF&=&\{g\in \RG\mid (g-1)V_i\subseteq V_{i-1},\text{ for }1\le i\le N\},\\
P_\CF&=&\{g\in \RG\mid (g-1)V \subseteq V_{N-1}\}.
\end{eqnarray*}
Of course,~$P_\CF$ does not depend on the whole flag~$\CF$, but~$U_\CF$ does: there is a bijection between maximal flags in~$V$ and maximal unipotent subgroups of~$\RG$.

Given now an ordered basis~$\ScB=(v_1,\ldots,v_N)$ of~$V$ we get a decomposition~$V=\bigoplus_{i=1}^N W_i$, where~$W_i=\langle v_i\rangle_F$ is the~$F$-linear span of~$v_i$. We set~$A_{ij}=\Hom_F(W_j,W_i)$ so that~$A=\bigoplus_{1\le i,j\le N}A_{ij}$, and define~$\boI_{ij}\in A_{ij}$ by~$\boI_{ij}(v_j)=v_i$. Thus saying that~$a=(a_{ij})$ is the matrix of some~$a\in A$ with respect to~$\ScB$, is the same as saying
\[
a=\sum_{1\le i,j\le N} a_{ij}\boI_{ij}.
\]
We also get a maximal flag~$\CF_{\ScB}$ by setting
\[
V_i=\bigoplus_{j=1}^i W_j=\langle v_1,\ldots,v_i\rangle_F.
\]
We denote the corresponding unipotent subgroup and mirabolic by~$U_{\ScB}$ and~$P_{\ScB}$ respectively.
Finally, we get a nondegenerate character~$\psi_{\ScB}$ of~$U_{\ScB}$, given by
\[
\psi_{\ScB}(u)=\psi_F\left(\sum_{i=1}^{N-1} u_{i,i+1}\right),
\]
where~$u\in U_{\ScB}$ and $(u_{ij})$ is the matrix of~$u$ with respect to the basis~$\ScB$. The same formula also defines a function~$\psi_{\ScB}$ on~$P_{\ScB}$ (though it is not a character).

The standard mirabolic subgroup, maximal unipotent subgroup, and nondegenerate character, are given by choosing~$\ScB$ to be the standard basis of~$V=F^N$.

\section{Minimax strata}\label{S:minimax}

In this section, for a minimax stratum~$[\FA,n,n-1,\bet]$, we examine
the relationship between the basis with respect to which~$\bet$ is in
companion form (and the associated mirabolic subgroup) and the
order~$\FA$.

We begin in the max setting (but not necessarily minimax). Suppose~$\bet\in A$ is such that~$E=F[\bet]$ is a field extension of~$F$ of maximal degree~$N$. We define the \emph{function}~$\psi_\bet$ of~$A$ by
\[
\psi_\bet(x)=\psi_F\circ\tr_{A/F}(\bet(x-1)), \quad\text{ for~$x\in A$.}
\]
There is an ordered basis~$\ScB=(v_1,\ldots,v_N)$ for~$V$ with respect to which
\[
\psi_\beta|{U_\ScB} = \psi_{\ScB}
\]
is the nondegenerate character associated to~$\ScB$. Indeed, there is, up to~$E^\times$-conjugacy, a unique maximal unipotent subgroup~$U$ such that~$\psi_\beta$ is trivial on the derived group~$U^{\der}$, and we have~$U_{\ScB}=U$. More explicitly, if~$v_1\in V$ is arbitrary, then putting
\[
v_j=\bet^{j-1}v_1,\text{ for }2\le j\le N,
\]
gives a basis as required, and every such basis arises in this way. With respect to the basis~$\ScB$, the matrix of~$\bet$ is the companion matrix of the minimum polynomial of~$\bet$. (See~\cite[Section~2]{BH98} for all of this.) The crucial (though trivial) observation is that we also have an equality of functions (not characters)
\begin{equation}\label{eqn:psiBonP}
\psi_\beta|{P_\ScB} = \psi_{\ScB}.
\end{equation}

Since~$E/F$ is maximal, there is a unique hereditary~$\Fo_F$-order~$\FA$ in~$A$ normalized by~$E^\times$; more precisely, it is given by the~$\Fo_F$-lattice chain~$\{\Fp_E^i\mid i\in\BZ\}$ so we have
\[
\FA=\{x\in A\mid x\Fp_E^i\subseteq\Fp_E^i,\hbox{ for all }i\in\BZ\}.
\]
It has~$\Fo_F$-period~$e(E/F)$, and consequently~$\nu_\FA(\bet)=\nu_E(\bet)$. We assume~$n=-\nu_\FA(\bet)>0$ so that~$[\FA,n,0,\bet]$ is a pure stratum, and then the restriction of the function~$\psi_\bet$ defines a character of~$U^{\left[\frac n2\right]+1}(\FA)$; moreover, by~\eqref{eqn:psiBonP}, we have an equality of characters
\begin{equation}\label{eqn:psiBonPcap}
\psi_\beta|{U^{\left[\frac n2\right]+1}(\FA)\cap P_\ScB} = \psi_{\ScB}.
\end{equation}
Now we have the following:

\begin{lem}\label{lem:maxintP}
Suppose~$[\FA,n,0,\bet]$ is a max pure stratum and put~$B=C_A(E)$ and~$\FB=\FA\cap B$, an~$\Fo_E$-hereditary order in~$B$. Let~$P$ be \emph{any} mirabolic subgroup of~$\RG$. Then, for any integer~$m\ge 1$, we have
\[
\left(U^m(\FB)U^n(\FA)\right)\cap P = U^n(\FA)\cap P.
\]
\end{lem}

\begin{proof}
Notice that actually~$B=E$ and~$\FB=\Fo_E$ in this situation. We pick an arbitrary uniformizer~$\varpi_E$ for~$E$. We prove that, for any~$m\ge 1$,
\[
\left(U^m(\FB)U^{m+1}(\FA)\right)\cap P = U^{m+1}(\FA)\cap P,
\]
and the result follows by iteration. This claim is equivalent to the following additive statement: setting~$\CP=\{p-1\mid p\in P\}$, we have
\[
\left(\Fp_E^m+\FP^{m+1}\right)\cap\CP = \FP^{m+1}\cap\CP,
\]
where~$\FP=\rad(\FA)$ as usual. So suppose~$x\in\Fp_E^m$ and~$y\in\FP^{m+1}$ are such that~$z:=x+y\in\CP$. In particular,~$z$ has eigenvalue~$0$, and the same is then true of~$\varpi_E^{-m}z\in\FO_E+\FP$, and of its image in~$\FA/\FP$. However, this image is in~$k_E\incl \FA/\FP$, and the only element of~$k_E$ with eigenvalue~$0$ is~$0$ itself. Thus~$\varpi_E^{-m}z\in\FP$ and~$z\in\FP^{m+1}$, as required.
\end{proof}

If~$\ScB=(v_1,\ldots,v_N)$ is an ordered basis, we put~$W_i=\langle
v_i\rangle_F$ and~$A_{ij}=\End_F(W_j,W_i)$, as before.
We say that~$\ScB$ is a \emph{splitting basis for~$\FA$} if
\[
\FA=\bigoplus_{1\le i,j\le N} \left(\FA\cap A_{ij}\right).
\]
In particular, any basis with respect to which~$\FA$ is in standard form is a splitting basis. Any permutation of a splitting basis~$\ScB$ is also a splitting basis; more generally, any basis obtained by a monomial change of basis from~$\ScB$ is a splitting basis.

Now we specialize to the case of a minimax stratum~$[\FA,n,n-1,\bet]$; in this case, we prove that any basis~$\ScB$ with respect to which~$\psi_\beta$ defines a character of~$U_\ScB$ is also a splitting basis for~$\FA$.

\begin{lem}\label{lem:minimaxsplit}
Suppose~$[\FA,n,n-1,\bet]$ is a minimax stratum and let~$\ScB=(v_1,\ldots,v_N)$ be an ordered basis for~$V$ such that~$\psi_\bet$ is trivial on~$U_{\ScB}^{\der}$. Then~$\ScB$ is a splitting basis for~$\FA$. Moreover, writing~$W_i=\langle v_i\rangle_F$ and~$A_{ij}=\End_F(W_j,W_i)$, for each~$1\le i,j\le N$ the lattice~$\FA\cap A_{ij}$ depends only on the depth~$n/e(\FA|\Fo_F)$ of the stratum.
\end{lem}

\begin{proof}
Note that the result is unaffected by multiplying the~$v_i$ by scalars so, identifying~$V$ with~$E$ via~$v_1\mapsto 1$, we may assume~$\ScB=(1,\bet,\ldots,\bet^{N-1})$.

We put~$n=-\nu_E(\beta)$,~$e=e(E/F)=e(\FA|\Fo_F)$ and~$f=f(E/F)$. Multiplication by~$n$ induces a bijection~$\BZ/e\BZ\to\BZ/e\BZ$. Thus, for each~$i=0,\ldots,e-1$, there is a unique integer~$r_i$, with~$0\le r_i<e$, such that~$nr_i\equiv -i\pmod{e}$. We write~$nr_i=d_i e - i$; then the fact that~$\bet$ is minimal implies that, for each~$i=0,\ldots,e-1$, the set
\[
\ScB'_{i}:=\{\varpi_F^{nk+d_i}\bet^{ek+r_i}\mid 0\le k\le f-1\}
\]
reduces to a basis for~$\Fp_E^{i}/\Fp_E^{i+1}$. Thus the ordered basis~$\ScB'$, obtained by ordering each~$\ScB'_{i}$ arbitrarily and then concatenating~$\ScB'_0,\ldots,\ScB'_{e-1}$,
is a splitting basis for~$\FA$ with respect to which~$\FA$ is in standard form. The change of basis matrix from~$\ScB'$ to~$\ScB=(1,\ldots,\bet^{ef-1})$ is monomial with entries from~$\varpi_F^{\BZ}$. Moreover, the~$(i,j)$ entry depends only on~$(i,j,n,e)$, and the result follows since~$n,e$ are determined by~$n/e$ as they are coprime.
\end{proof}

\section{Jacquet's conjecture}

Finally, we prove that if~$\pi_1,\pi_2$ are a pair of supercuspidal
representations of~$\RG$ containing minimax strata, then they have a
special pair of Whittaker functions. When~$N$ is prime, any
positive depth supercuspidal of minimal depth in its twist class
contains a minimax stratum, so Jacquet's conjecture will follow from
Propositions~\ref{prop:SPimpliesJ} and~\ref{prop:redmin}, together
with the depth zero case from~\cite[Corollary~1.7]{JNS13}.

At this point, we need to recall a little more on the construction of
the supercuspidal representations of~$\RG$. Since it is all we will need,
we only give definitions for representations which contain a minimax
stratum. Thus, in a slightly different language, we are recounting the
constructions of Carayol~\cite{C84}.

Let~$[\FA,n,n-1,\bet]$ be a minimax stratum,
with~$E=F[\bet]$. Associated to the simple stratum~$[\FA,n,0,\bet]$,
we have the following compact open subgroups of~$\RG$, contained in
the normalizer~$\FK(\FA)$:
\begin{eqnarray*}
&&H^1=H^1(\bet,\FA)=U^1(\Fo_E)U^{\left[\frac n2\right]+1}(\FA),\\
&&J^1=J^1(\bet,\FA)=U^1(\Fo_E)U^{\left[\frac {n+1}2\right]}(\FA),\\
&&\BJ=\BJ(\bet,\FA)=E^\times U^{\left[\frac {n+1}2\right]}(\FA).
\end{eqnarray*}
A \emph{simple character} is then a character of~$H^1$ which extends
the character~$\psi_\bet$ of~$U^{\left[\frac n2\right]+1}(\FA)$. Given
such a simple character~$\theta$, there is a unique irreducible
representation~$\eta$ of~$J^1$ which contains~$\theta$ on restriction
to~$H^1$ (indeed, it is a multiple). An \emph{extended maximal simple
  type} is then an irreducible representation~$\Lambda$ of~$\BJ$ which
extends~$\eta$. Given such a maximal simple type, the representation
\[
\ind_{\BJ}^{\RG}\Lambda
\]
is irreducible and supercuspidal and, moreover, every irreducible
supercuspidal representation containing~$\theta$ arises in this way.

Any irreducible supercuspidal representation~$\pi$ containing a
minimax stratum~$[\FA,n,n-1,\bet]$ contains some simple
character~$\theta$ associated to a simple stratum~$[\FA,n,0,\bet']$,
with~$[\FA,n,n-1,\bet']$ minimax and equivalent
to~$[\FA,n,n-1,\bet]$. Thus~$\pi$ also contains~$[\FA,n,n-1,\bet']$
and we may assume~$\bet'=\bet$.
%

To prove the main result of this paper, recall that we need to exhibit
a special pair of Whittaker functions for two supercuspidal
representations of a specific form.  In \cite{PS08}, Whittaker
functions are constructed which carry the properties that we need.  We
record the result in a form that does not require additional
background.  Recall first (see \S\ref{S:minimax}) that to~$\beta$ we
associate a basis~$\ScB$ and a unipotent subgroup~$U_{\ScB}$ such that
\[
\psi_\beta|{U_\ScB} = \psi_{\ScB}
\]
is the nondegenerate character associated to~$\ScB$.

\begin{prop}[{\cite[Theorem~5.8]{PS08}}]\label{prop:pswhittaker}
There exists a Whittaker function~$W$ for~$\pi$ such that~$\Supp(W)
\subset U_{\ScB} \BJ$ and such that, for~$g \in P_{\ScB}$,
\[
W(g)=\begin{cases}
\psi_{\ScB}(u) \theta(h) &\text{ if }g = uh \in (J^1_i\cap U_{\ScB})H^1_i,\\
0 &\text{ otherwise.}
\end{cases}
\]
\end{prop}

Our main result is:

\begin{prop}\label{prop:minimax}
For~$i=1,2$, let~$\pi_i$ be a (positive depth) unitarizable
supercuspidal representation of~$\RG$ containing a minimax
stratum. Suppose that~$\pi_1,\pi_2$ have the same depth and the same
central character. Then~$\pi_1,\pi_2$ have a special pair of
Whittaker functions.
\end{prop}
\begin{proof}
For~$i=1,2$, let~$[\FA_i,n_i,n_i-1,\bet_i]$ be a minimax stratum of
period~$e_i=e(\FA_i|\Fo_F)$ contained in~$\pi_i$, with the property
that~$\pi_i$ also contains a simple character~$\theta_i$
of~$H^1_i=H^1(\bet_i,\FA_i)$. Since the representations have the same
depth we have~$n_1/e_1=n_2/e_2$; since they are minimax, we
have~$\gcd(n_i,e_i)=1$ and we may write~$n=n_1=n_2$ and~$e=e_1=e_2$.

Fix~$v\in V$ and let~$g\in \RG$ be the change of basis matrix from the
basis~$\ScB=(v,\bet_1 v,\ldots,\bet_1^{N-1} v)$ to~$(v,\bet_2
v,\ldots,\bet_2^{N-1} v)$. Then, replacing the
stratum~$[\FA_2,n,n-1,\bet_2]$ and the simple character~$\theta_2$ by
their conjugates by~$g$, we can assume that \emph{both}~$\bet_i$ are
in companion matrix form with respect to~$\ScB$,
i.e.~$\bet_1^{j-1}v=\bet_2^{j-1}v$, for~$1\le j\le N$. By
Lemma~\ref{lem:minimaxsplit}, the hereditary orders
coincide:~$\FA_1=\FA_2=\FA$. By Lemma~\ref{lem:maxintP}, we
have~$H^1_i\cap P_{\ScB}=U^{\left[\frac n2\right]+1}(\FA)\cap
P_{\ScB}$ so, by~\eqref{eqn:psiBonPcap}, we have
\[
\theta_i|{H^1_i\cap P_{\ScB}}=\psi_{\ScB},
\]
independent of~$i$. Moreover, we then have~$\Hom_{H^1_i\cap
  U_{\ScB}}(\theta_i,\psi_{\ScB})\ne 0$.

We abbreviate~$J^1_i=J^1(\bet_i,\FA)$
and~$\BJ_i=\BJ(\beta_i,\FA)$, and denote by~$\eta_i$ the unique
irreducible representation of~$J^1_i$ containing~$\theta_i$. Then
we also have~$\Hom_{J^1_i\cap  U_{\ScB}}(\eta_i,\psi_{\ScB})\ne 0$,
by~\cite[Theorem~2.6]{PS08}.
Writing~$\pi_i=\ind_{\BJ_i}^\RG\Lambda_i$, with~$\Lambda_i$
extending~$\eta_i$, we have~$\BJ_i\cap U_{\ScB}=J^1_i\cap U_{\ScB}$
and we see that~$\Hom_{\BJ_i\cap U_{\ScB}}(\Lambda_i,\psi_{\ScB})\ne
0$. Thus we have a Whittaker function~$W_i$ as constructed
in Proposition \ref{prop:pswhittaker} relative to the pair~$(U_\ScB,\psi_{\ScB})$ and these
coincide on~$P_{\ScB}$ since, for~$g\in P_{\ScB}$,
\[
W_i(g)=\begin{cases}
\psi_{\ScB}(g) &\text{ if }g\in (J^1_i\cap U_{\ScB})(H^1_i\cap P_{\ScB}),\\
0 &\text{ otherwise.}
\end{cases}
\]
Putting~$\BK=\FK(\FA)$, both~$W_i$ are~$\BK$-special
(see~\cite[Lemma~4.2]{JNS13}) so we have found a special pair of
Whittaker functions.
\end{proof}

\begin{cor}\label{cor:Jprime}
For~$N$ prime, Conjecture~$\CJ(N,[\frac N2])$ is true.
\end{cor}

\begin{proof}
Let~$\pi_1,\pi_2$ be unitarizable supercuspidal representations of~$\RG$
satisfying hypothesis~$\CH_{\le [N/2]}$, and of minimal depth in their
twist classes. In particular, from hypothesis~$\CH_1$, the
representations~$\pi_1,\pi_2$ have the same central character and
depth. If both have depth zero then they are equivalent
by~\cite[Corollary~1.7]{JNS13}, so we assume they are of positive depth.
Since~$N$ is prime,~$\pi_i$ contains a minimax
stratum~$[\FA_i,n_,n_i-1,\bet_i]$, for~$i=1,2$. Since the
representations~$\pi_1,\pi_2$ have the same depth, we
have~$n_1/e_1=n_2/e_2$ and, since~$n_i,e_i$ are coprime, in
fact~$n_1=n_2$. But now Proposition~\ref{prop:minimax} implies
that~$\pi_1,\pi_2$ have a special pair of Whittaker functions, and
Proposition~\ref{prop:SPimpliesJ} implies
that~$\pi_1\simeq\pi_2$. Thus Conjecture~$\CJ_0(N,[\frac N2])$ is
true, and the result now follows from Proposition~\ref{prop:redmin}.
\end{proof}


\begin{thebibliography}{}
\bibitem[AL14] {AL14}
M. Adrian and B. Liu,
{\it Some results on simple supercuspidal representations of $\GL_n(F)$,}
submitted, 2014.
\bibitem[B87]{B87}
C. Bushnell,
{\it Hereditary orders, Gauss sums and supercuspidal representations of $\GL_N$.}
J. Reine Angew. Math. \textbf{375/376} (1987), 184--210.
\bibitem[BK93] {BK93}
C. Bushnell and P. Kutzko,
{\it The admissible dual of $\GL_N$ via restriction to compact open subgroups.} Annals of Mathematics Studies, \textbf{129}. Princeton University Press, Princeton, NJ, 1993. xii+313 pp.
\bibitem[BH98] {BH98}
C. Bushnell and G. Henniart,
{\it Supercuspidal representations of $\GL_n$: explicit Whittaker functions.}
J. Algebra \textbf{209} (1998), 270--287.
\bibitem[BH14] {BH14}
C. Bushnell and G. Henniart,
{\it Langlands parameters for epipelagic representations of $\GL_n$.}
Math. Ann. \textbf{358} (2014), no. 1--2, 433--463.
\bibitem[Ch96] {Ch96}
J.-P. Chen,
{\it Local factors, central characters, and representations of the general linear group over non-Archimedean local fields.}
Thesis, Yale University, 1996.
\bibitem[C84] {C84}
H. Carayol,
{\it Repr\'esentations cuspidales du groupe lin\'eaire.}
Ann. Sci. \'Ecole Norm. Sup. (4) \textbf{17} (1984) 191--225.
\bibitem[Ch06] {Ch06}
J.-P. Chen,
{\it The $n \times (n-2)$ local converse theorem for $\GL(n)$ over a $p$-adic field.}
J. Number Theory \textbf{120} (2006), no. 2, 193--205.
\bibitem[CPS99]{CPS99}
Cogdell, J.~W., Piatetskii-Shapiro, I.~I.: {Converse theorems for {${\rm
  GL}_n$}. {II}}, J. Reine Angew. Math. \textbf{507} (1999), 165--188.
\bibitem[JPSS83] {JPSS83}
H. Jacquet, I. Piatetskii-Shapiro and J. Shalika,
{\it Rankin--Selberg convolutions.}
Amer. J. Math. \textbf{105} (1983), 367--464.
\bibitem[JNS13] {JNS13}
D. Jiang, C. Nien and S. Stevens,
{\it Towards the Jacquet Conjecture on the Local Converse Problem for $p$-Adic $\GL_n$.}
To appear in Journal of the European Mathematical Society. 2013.
\bibitem[KM90] {KM90}
P. Kutzko and D. Manderscheid
{\it On the supercuspidal representations of $\GL_N$, $N$ the product of two primes.}
Ann. Sci. École Norm. Sup. (4) \textbf{23} (1990), no. 1, 39--88.
\bibitem[LR03] {LR03}
J. Lansky and A. Raghuram,
{\it A remark on the correspondence of representations between
$\GL(n)$ and division algebras.}
Proc. Am. Math. Soc. \textbf{131} (5) (2003) 1641--1648.
\bibitem[PS08] {PS08}
V. Paskunas and S. Stevens,
{\it On the realization of maximal simple types and epsilon factors of pairs.}
Amer. J. Math. \textbf{130} (5) (2008), 1211--1261.
\bibitem[S84] {S84}
F. Shahidi,
{\it Fourier transforms of intertwining operators and Plancherel measures for $\GL(n)$.}
Amer. J. Math. \textbf{106} (1984), no. 1, 67--111.
\bibitem[X13] {X13}
P. Xu,
{\it A remark on the simple cuspidal representations of $\GL_n$.}
Preprint. 2013.


\end{thebibliography}
\end{document}